\documentclass[12pt,twoside]{amsart}
\usepackage{amsmath}
\usepackage{amsthm}
\usepackage{amsfonts}
\usepackage{amssymb}
\usepackage{latexsym}
\usepackage[all]{xy}
\usepackage{extarrows}

\date{}
\allowdisplaybreaks[4] \footskip=15pt
\renewcommand{\uppercasenonmath}[1]{}

\numberwithin{equation}{section} \theoremstyle{plain}
\newtheorem*{thm*}{Main Theorem}
\newtheorem{thm}{Theorem}[section]
\newtheorem{cor}[thm]{Corollary}
\newtheorem*{cor*}{Corollary}
\newtheorem{lem}[thm]{Lemma}
\newtheorem*{lem*}{Lemma}
\newtheorem{prop}[thm]{Proposition}
\newtheorem*{prop*}{Proposition}
\newtheorem{rem}[thm]{Remark}
\newtheorem*{rem*}{Remark}

\newtheorem*{exa*}{Example}
\newtheorem{df}[thm]{Definition}
\newtheorem*{df*}{Definition}

\newtheorem*{conj*}{Conjecture}
\newtheorem*{ack*}{ACKNOWLEDGEMENTS}

\newcommand{\pf}{\noindent\begin {proof}}
\newcommand{\epf}{\end{proof}}

\newcommand{\Ext}{\mbox{\rm Ext}}

\newcommand{\Hom}{\mbox{\rm Hom}}
\newcommand{\Tor}{\mbox{\rm Tor}}
\newcommand{\im}{\mbox{\rm im}}
\newcommand{\coker}{\mbox{\rm coker}}
\begin{document}
\begin{center}
{\bf Homological aspects of $n$-trivial extensions of rings}

\vspace{0.5cm} Lixin Mao\\
%\bigskip
School of Mathematics and Physics, Nanjing Institute of Technology,\\ Nanjing 211167, China\\
E-mail: maolx2@hotmail.com \\
\end{center}
%\begin{figure}[b]
%\rule[-2.5truemm]{5cm}{0.1truemm}\\[2mm]
%{ }
%\end{figure}

%\begin{figure}[b]
%\rule[-2.5truemm]{5cm}{0.1truemm}\\[2mm]
%{ }
%\end{figure}
\bigskip
\centerline { \bf  Abstract}
 \bigskip
\leftskip10truemm \rightskip10truemm
 \noindent
Let $R\ltimes_{n}M$ be the $n$-trivial extension of a ring $R$ by an $R$-bimodule $M$. We first characterize projective, injective, flat  and finitely generated modules over $R\ltimes_{n}M$. Then we describe when the ring $R\ltimes_{n}M$ is left perfect (Noetherian, coherent, Artinian, hereditary, Kasch). Finally, we establish some homological formulas of $n$-trivial extensions of rings.\\
\vbox to 0.3cm{}

\leftskip0truemm \rightskip0truemm
\bigskip
\section { \bf Introduction}
Let $R$ be an associative ring and $M$ an $R$-bimodule. Then the Cartesian product $R \times M$, with the natural addition and the multiplication given by $(r_{1}, m_{1})(r_{2}, m_{2}) = (r_{1}r_{2}, r_{1}m_{2}+m_{1}r_{2})$, becomes a ring. This ring is called the \emph{trivial  extension} of the ring $R$ by the bimodule $M$, and denoted
by $R\ltimes M$. The notion of trivial extension of a ring by a bimodule is an important extension of rings and has played a crucial role in ring theory and homological algebra \cite{DMZ, FGR, PR, R}. In \cite{ABFS}, Anderson, Bennis, Fahid and  Shaiea introduced the concept of $n$-trivial extension of rings for a family $(M_{i})_{i=1}^{n}$ of $R$-bimodules, denoted by $R\ltimes_{n} M_{1}\ltimes \cdots\ltimes M_{n}$ and  obtained many ring-theoretic properties of this kind of rings. Later, Benkhadra, Bennis and Garc\'{\i}a Rozas studied some basic module-theoretic properties of $R\ltimes_{n} M_{1}\ltimes \cdots\ltimes M_{n}$ \cite{BBG}.

In the present paper, we will focus on a kind of special $n$-trivial extension $R\ltimes_{n}M$ for an $R$-bimodule $M$ (see Definition \ref{df: 2.1}), which is not only a generalization of  the classical trivial extension  $R\ltimes M$, but also a quotient ring of the tensor ring $T_{R}(M)$ in the sense of \cite{Co}. Such an $n$-trivial extension $R\ltimes_{n}M$ admits concise and interesting homological properties as well as classical trivial  extensions \cite{M}.

In Section 2,  we first give the equivalent category $\Omega$ of $R\ltimes_{n}M$-Mod (see Proposition \ref{prop: 2.4}). Then we characterize  projective, injective and flat modules over $R\ltimes_{n}M$ (see Theorems \ref{thm: 2.8}, \ref{thm: 2.9} and \ref{thm: 2.10}). Finally, we describe when a left $R\ltimes_{n}M$-module is finitely generated or finitely presented (see Theorem \ref{thm: 2.12}). In Section 3, we characterize when the ring $R\ltimes_{n}M$ is a left perfect (Noetherian, coherent, Artinian, hereditary, self-injective, Kasch) ring (see Theorems \ref{thm: 3.1}, \ref{thm: 3.2} and \ref{thm: 3.4}). Section 4 is devoted to several isomorphism formulas of homological groups over $R\ltimes_{n}M$ (see Theorems \ref{thm: 4.2}, \ref{thm: 4.5} and \ref{thm: 4.8}). As an application, some estimates of homological dimensions of modules over $R\ltimes_{n}M$ are established (see Corollaries \ref{cor: 4.3}, \ref{cor: 4.6} and \ref{cor: 4.9}).

Throughout this paper, all rings are nonzero associative rings with identity and all modules are unitary. We write $R$-Mod (resp. Mod-$R)$  for the category of left (resp. right) $R$-modules. $_RX$ (resp. $X_{R}$) denotes a left (resp. right) $R$-module. $pd(_{R}X)$, $id(_{R}X)$ and $fd(_{R}X)$ denote the projective, injective and flat dimensions of $_{R}X$ respectively.
\bigskip
\section {\bf Modules over $n$-trivial extensions of rings}
\bigskip
\begin{df}\label{df: 2.1} {\rm Let $M$ be an $R$-bimodule and $n\geq 1$. Define a multiplication in the direct sum $\coprod_{i=0}^{n}M^{\otimes_{R}i}$ of Abelian groups,
where $M^{\otimes_{R}0} = R$, $M^{\otimes_{R}1}=M$ and $M^{\otimes_{R}i+1} =  M\otimes_{R}M^{\otimes_{R}i}$, by
$(m_{0}, ..., m_{n})(m_{0}^{'}, ..., m_{n}^{'}) = (\sum_{i+j=k} m_{i}\otimes m_{j}^{'})_{0\leq k\leq n}$. Then $\coprod_{i=0}^{n}M^{\otimes_{R}i}$ becomes a ring. This ring is called an \emph{$n$-trivial extension} of the ring $R$ by $M$, denoted by $R\ltimes_{n} M$.}
\end{df}
\begin{rem}\label{rem: 2.2}{\rm (1) Notice that $R\ltimes_{1} M$ is nothing but the classical trivial extension $R\ltimes M$.

(2) Recall that a \emph{tensor ring} of an $R$-bimodule $M$ is $T_{R}(M) = \coprod_{i=0}^{\infty}M^{\otimes_{R}i}$ \cite{Co}. It is easy to verify that there is a ring isomorphism $R\ltimes_{n} M\cong T_{R}(M)/\coprod_{i=n+1}^{\infty}M^{\otimes_{R}i}$. In particular, if $M^{\otimes_{R}n+1}=0$, then $R\ltimes_{n} M\cong T_{R}(M)$.}
\end{rem}
In order to investigate modules over $R\ltimes_{n} M$, we introduce the following category $\Omega$: whose objects are couples $(X, f)$ with $X$ a left $R$-module and $f\in\Hom_{R}(M\otimes_{R}X, X)$ such that the composition of $$M^{\otimes_{R}n+1}\otimes_{R} X\stackrel{M^{\otimes_{R}n}\otimes f}\longrightarrow M^{\otimes_{R}n}\otimes_{R} X\longrightarrow\cdots\longrightarrow  M\otimes_{R}  M\otimes_{R} X\stackrel{ M\otimes f}\longrightarrow  M\otimes_{R} X\stackrel{f}\longrightarrow X$$ is zero and a morphism $(X,f) \rightarrow (Y,g)$ is  $\gamma\in\Hom_{R}(X, Y)$ such that the following diagram commutes. $$\xymatrix{M\otimes_{R}X\ar[d]_{f}\ar[r]^{M\otimes\gamma}&M\otimes_{R}Y\ar[d]_{g}\\X\ar[r]^{\gamma}&Y}$$
A sequence $(X_{1},f_{1})\stackrel{\gamma_{1}}\longrightarrow (X_{2},f_{2})\stackrel{\gamma_{2}}\longrightarrow(X_{3},f_{3})$ in $\Omega$ is exact if and only if the underlying sequence $X_{1}\stackrel{\gamma_{1}}\longrightarrow X_{2}\stackrel{\gamma_{2}}\longrightarrow X_{3}$ in $R$-Mod is exact.

In view of the adjointness relation, the category $\Omega$ is equivalent  to the category $\Upsilon$: whose objects are couples $[Y,g]$ with $Y$ a left $R$-module and $g\in\Hom_{R}(Y, \Hom_{R}(M,Y))$ such that the composition of {\small$$Y\stackrel{g}\longrightarrow \Hom_{R}(M,Y)\rightarrow \Hom_{R}(M,\Hom_{R}(M,Y))\rightarrow\cdots\rightarrow\Hom_{R}^{n}(M,Y)\stackrel{{\rm Hom}_{R}^{n}(M,g)}\longrightarrow  \Hom_{R}^{n+1}(M,Y)$$} is zero, where $\Hom_{R}^{0}(M,Y)=Y, \Hom_{R}^{1}(M,Y)=\Hom_{R}(M,Y), \cdots,
\Hom_{R}^{n}(M,Y)=\Hom_{R}(M,\Hom_{R}^{n-1}(M,Y))$, and a morphism $[X,f] \rightarrow [Y,g]$ is  $\varphi\in\Hom_{R}(X, Y)$ such that the following diagram commutes. $$\xymatrix{X\ar[d]_{f}\ar[rr]^{\varphi}&&Y\ar[d]_{g}\\\Hom_{R}(M,X)\ar[rr]^{{\rm Hom}_{R}(M,\varphi)}&&\Hom_{R}(M,Y)}$$
A sequence $[Y_{1},g_{1}]\stackrel{\varphi_{1}}\longrightarrow [Y_{2},g_{2}]\stackrel{\varphi_{2}}\longrightarrow [Y_{3},g_{3}]$ in $\Upsilon$ is exact if and only if the underlying sequence $Y_{1}\stackrel{\varphi_{1}}\longrightarrow Y_{2}\stackrel{\varphi_{2}}\longrightarrow Y_{3}$ in $R$-Mod is exact.

We define some functors as follows.

The functor $\textbf{T}: R$-Mod $\longrightarrow \Omega$ is given, for every object $X\in R$-Mod, by $\textbf{T}(X) =(\coprod_{i=0}^{n}(M^{\otimes_{R}i}\otimes_{R}X), \mu)$, with $\mu(x_{0}, x_{1}, x_{2},\cdots, x_{n})=(0, x_{0}, x_{1},\cdots, x_{n-1})$  and for morphisms by $\textbf{T}(\alpha)=\coprod_{i=0}^{n}(M^{\otimes_{R}i}\otimes\alpha)$.

The functor $\textbf{U}:\Omega\longrightarrow R$-Mod is given, for every object $(X,f)\in \Omega$, by $\textbf{U}(X,f) =X$ and for morphisms by $\textbf{U}(\alpha)=\alpha$. Since the category $\Upsilon$ is  equivalent  to $\Omega$, $\textbf{U}$ can also be interpreted as the functor $\textbf{U}:\Upsilon\longrightarrow R$-Mod: for every object $[Y,g]\in \Upsilon$, $\textbf{U}[Y,g] =Y$ and for morphisms, $\textbf{U}(\beta)=\beta$.

The functor $\textbf{Z}: R$-Mod $\longrightarrow \Omega$ is given, for every object $X\in R$-Mod, by $\textbf{Z}(X)=(X,0)$ and for morphisms by $\textbf{Z}(\alpha)=\alpha$. Since the category $\Upsilon$ is  equivalent  to $\Omega$, $\textbf{Z}$ can also be interpreted as the functor $\textbf{Z}: R$-Mod $\longrightarrow\Upsilon$: for every object  $Y\in R$-Mod, $\textbf{Z}(Y)=[Y,0]$ and for morphisms, $\textbf{Z}(\beta)=\beta$.

The functor $\textbf{C}:\Omega\longrightarrow R$-Mod is given, for every object $(X,f)\in \Omega$, by $\textbf{C}(X,f) =\coker(f)$ and for morphisms by $\textbf{C}(\alpha)=$ the induced morphism.

The functor $\textbf{H}: R$-Mod $\longrightarrow \Upsilon$ is given, for every object  $Y\in R$-Mod, by $\textbf{H}(Y)=[\coprod_{i=0}^{n}\Hom_{R}^{i}(M,Y),\nu]$,  with $\nu(x_{0}, x_{1}, x_{2},\cdots, x_{n})=(x_{1}, x_{2},\cdots, x_{n},0)$ and for morphisms by $\textbf{H}(\beta)=\coprod_{i=0}^{n}\Hom^{i}_{R}(M,\beta)$.

The functor $\textbf{K}: \Upsilon\longrightarrow R$-Mod is given, for every object $[Y,g]\in \Upsilon$, by $\textbf{K}[Y,g] =\ker(g)$ and for morphisms by $\textbf{K}(\beta)=$ the induced morphism.

There are analogous functors for right modules.
\begin{prop}\label{prop: 2.3} $(\textbf{T}, \textbf{U})$, $(\textbf{U}, \textbf{H})$, $(\textbf{C}, \textbf{Z})$ and $(\textbf{Z}, \textbf{K})$ are adjoint pairs  such that $\textbf{C} \textbf{T}= id_{R{\rm-Mod}}$, $\textbf{U} \textbf{Z}= id_{R{\rm-Mod}}$ and $\textbf{K} \textbf{H}= id_{R{\rm-Mod}}$. Consequently, $\textbf{C}$ and $\textbf{T}$ are right exact, $\textbf{K}$ and $\textbf{H}$ are left exact.
\end{prop}
\begin{proof}It is straightforward (c.f. \cite[Proposition 1.3]{FGR}).
\end{proof}
The following result is virtually due to \cite[Theorem 2.2]{BBG}. We include the proof for the reader's convenience.
\begin{prop}\label{prop: 2.4}The category $R\ltimes_{n}M$-Mod is equivalent to the category $\Omega$.
\end{prop}
\begin{proof}Let $X$ be a left $R\ltimes_{n}M$-module.  Then $X$ may be viewed as a left $R$-module by the obvious ring homomorphism $R\rightarrow R\ltimes_{n}M$.
Define $f_{X}:  M\otimes_{R}X \rightarrow X$ by $f_{X}(m\otimes x)=(0,m,0,\cdots,0)x$ for any $m\otimes x\in  M\otimes_{R}X$. Note that $$f_{X}(r(m\otimes x))=(0, rm, 0,\cdots)x=(r, 0, 0,\cdots)(0, m, 0,\cdots)x=rf_{X}(m\otimes x)$$ for any $r\in R$. So $f_{X}$ is a left $R$-homomorphism. Also $f_{X}(M\otimes f_{X})\cdots(M^{\otimes_{R}n}\otimes f_{X})(m_{1}\otimes m_{2}\otimes m_{3}\cdots\otimes m_{n+1}\otimes x)=(0,m_{1},0,\cdots,0)(0,m_{2},0,\cdots,0)\cdots(0,m_{n+1},0,\cdots,0)x=0.$ So we can define the functor $\textbf{F}: R\ltimes_{n}M$-Mod $\longrightarrow \Omega$ by $\textbf{F}(X)=(X, f_{X})$. Let $\alpha: A\rightarrow B$ be a left $R\ltimes_{n}M$-homomorphism. Then $\alpha$ may be viewed as a left $R$-homomorphism. Note that $\alpha f_{A}(m\otimes a)=\alpha((0,m,0,\cdots,0)a)=(0,m,0,\cdots,0)\alpha(a)=f_{B}(M\otimes\alpha)(m\otimes a)$ for any $m\otimes a\in M\otimes_{R} A$. Hence $\alpha f_{A}=f_{B}(M\otimes\alpha)$. Thus $\alpha\in\Hom_{\Omega}((A,f_{A}),(B,f_{B}))$. So we can define $\textbf{F}(\alpha)=\alpha$.

Conversely, let $(Y,h)\in\Omega$. For any $(m_{0}, ..., m_{n})\in R\ltimes_{n}M$ and $y\in Y$, define $$(m_{0}, ..., m_{n})y=m_{0}y+\sum\limits_{i=1}^{n}h(M\otimes h)\cdots(M^{\otimes_{R}i-1}\otimes h)(m_{i}\otimes y).$$ Then $Y$ is clearly a left $R\ltimes_{n}M$-module. We define the functor $\textbf{G}: \Omega\longrightarrow R\ltimes_{n}M$-Mod by $\textbf{G}(Y,h)=Y$. Let $\gamma\in\Hom_{\Omega}((Y,h),(G,g))$ and $y\in Y$. Then $$\gamma((m_{0}, ..., m_{n})y)=\gamma(m_{0}y+\sum\limits_{i=1}^{n}h(M\otimes h)\cdots(M^{\otimes_{R}i-1}\otimes h)(m_{i}\otimes y))=(m_{0}, ..., m_{n})\gamma(y).$$ So $\gamma$ is a left $R\ltimes_{n}M$-homomorphism. Thus we can define  $\textbf{G}(\gamma)=\gamma$.

It is easy to check that $\textbf{G}\textbf{F}$ is naturally equivalent to $id_{R\ltimes_{n}M{\rm-Mod}}$ and $\textbf{F}\textbf{G}$ is naturally equivalent to $id_{\Omega}$. So the category $R\ltimes_{n}M$-Mod is equivalent to the category $\Omega$.
\end{proof}
\begin{rem}\label{rem: 2.5}{\rm Let $X$ be a left $R$-module. It is clear that the left $R\ltimes_{n}M$-module $(R\ltimes_{n} M)\otimes_{R}X$ corresponds to $\textbf{T}(X)$ in $\Omega$ by Proposition \ref{prop: 2.4}. In particular, the  regular left $R\ltimes_{n}M$-module $R\ltimes_{n}M$ corresponds to  $\textbf{T}(R)$. Similarly, the left $R\ltimes_{n}M$-module $\Hom_{R}(R\ltimes_{n}M,X)$ corresponds to $\textbf{H}(X)$ in $\Upsilon$.

In the rest of the paper, we will identify  $R\ltimes_{n}M$-Mod  with $\Omega$ and $\Upsilon$.}
\end{rem}
\begin{lem}\label{lem: 2.6}Let $(X,f)$ and $[Y,g]$ be  left $R\ltimes_{n}M$-modules.
\begin{enumerate}\item  There is an exact sequence $0\rightarrow (\coprod_{i=1}^{n}(M^{\otimes_{R}i}\otimes_{R}X), h)\rightarrow\textbf{T}(X)\rightarrow(X,f)\rightarrow 0$.\item There is an exact sequence  $0\rightarrow [Y,g]\rightarrow\textbf{H}(Y)\rightarrow[\coprod_{i=1}^{n}\Hom^{i}_{R}(M,Y),d]\rightarrow 0$.\item There exist  exact sequences $$0\rightarrow (\im(f),u_{1})\rightarrow(X,f)\rightarrow\textbf{Z}(\coker(f))\rightarrow 0,$$ $$0\rightarrow(\im(u_{1}),u_{2})\rightarrow(\im(f),u_{1})\rightarrow\textbf{Z}(\coker(u_{1}))\rightarrow 0,$$$$\cdots\cdots$$ $$0\rightarrow\textbf{Z}(\im(u_{n-1}))\rightarrow(\im(u_{n-2}),u_{n-1})\rightarrow\textbf{Z}(\coker(u_{n-1}))\rightarrow 0.$$
\end{enumerate}
\end{lem}
\begin{proof}(1) Define $\rho: \coprod_{i=0}^{n}(M^{\otimes_{R}i}\otimes_{R}X)\rightarrow X$ by $$\rho(x_{0},x_{1},x_{2},\cdots,x_{n})=x_{0}+\sum_{i=1}^{n} f(M\otimes f)\cdots(M^{\otimes_{R}i-1}\otimes f)(x_{i}).$$ It is easy to check that $\rho$ is an epic left $R$-homomorphism such that $\rho\mu=f(M\otimes\rho)$. So $\rho\in\Hom_{R\ltimes_{n}M}(\textbf{T}(X),(X,f))$.

Define $h: M\otimes_{R}\coprod_{i=1}^{n}(M^{\otimes_{R}i}\otimes_{R}X)\rightarrow \coprod_{i=1}^{n}(M^{\otimes_{R}i}\otimes_{R}X)$ by $h(m\otimes(y_{1},y_{2},\cdots,y_{n}))=(-(M\otimes f)
\cdots(M^{{\otimes_{R}n}}\otimes f)(m\otimes y_{n}), m\otimes y_{1}-(M^{{\otimes_{R}2}}\otimes f)\cdots(M^{{\otimes_{R}n}}\otimes f)(m\otimes y_{n}), \cdots, m\otimes y_{n-1}-(M^{{\otimes_{R}n}}\otimes f)(m\otimes y_{n})).$ It is easy to verify that $h$ is a left $R$-homomorphism such that $h(M\otimes h)\cdots(M^{\otimes_{R}n}\otimes h)=0$. Thus $(\coprod_{i=1}^{n}(M^{\otimes_{R}i}\otimes_{R}X), h)$ is a left $R\ltimes_{n}M$-module.

Define $\lambda: \coprod_{i=1}^{n}(M^{\otimes_{R}i}\otimes_{R}X)\rightarrow \coprod_{i=0}^{n}(M^{\otimes_{R}i}\otimes_{R}X)$ by $$\lambda(y_{1},y_{2},\cdots,y_{n})=(-f(y_{1}),y_{1}-(M\otimes f)(y_{2}),\cdots,y_{n-1}-(M^{\otimes_{R}n-1}\otimes f)(y_{n}), y_{n}).$$ It is easy to check that  $\lambda$ is a monic  left $R$-homomorphism such that $\mu(M\otimes\lambda)=\lambda h$. Thus $\lambda\in\Hom_{R\ltimes_{n}M}((\coprod_{i=1}^{n}(M^{\otimes_{R}i}\otimes_{R}X), h), \textbf{T}(X))$.
Note that $$\rho\lambda(y_{1},y_{2},\cdots,y_{n})=\rho(-f(y_{1}),y_{1}-(M\otimes f)(y_{2}),\cdots,y_{n-1}-(M^{\otimes_{R}n-1}\otimes f)(y_{n}), y_{n})=0.$$ So $\im(\lambda)\subseteq\ker(\rho)$. Let $(x_{0},x_{1},x_{2},\cdots,x_{n})\in\ker(\rho)$. Then $$x_{0}+\sum_{i=1}^{n} f(M\otimes f)\cdots(M^{\otimes_{R}i-1}\otimes f)(x_{i})=0.$$ Thus  $(x_{0},x_{1},x_{2},\cdots,x_{n})=\lambda(x_{1}+\sum\limits_{i=2}^{n}(M\otimes f)\cdots(M^{\otimes_{R}i-1}\otimes f)(x_{i}),x_{2}+\sum_{i=3}^{n}\limits(M^{\otimes_{R}2}\otimes f)\cdots(M^{\otimes_{R}i-1}\otimes f)(x_{i}),\cdots,x_{n}).$ Hence $\ker(\rho)\subseteq\im(\lambda)$ and so $\im(\lambda)=\ker(\rho)$. Therefore we get the exact sequence  $0\rightarrow\textbf{T}(M\otimes_{R}X)\stackrel{\lambda}\rightarrow\textbf{T}(X)\stackrel{\rho}\rightarrow(X,f)\rightarrow 0$.

The proof of (2) is dual.

(3) There exists the following commutative diagram with exact rows: $$\xymatrix{&M\otimes_{R}\im(f)\ar[r]\ar[d]_{u_{1}}&M \otimes_{R}X\ar[d]_{f}\ar[r]&M \otimes_{R}\coker(f)\ar[d]_{0}\ar[r]&0\\0\ar[r]&\im(f)\ar[r]&X\ar[r]&\coker(f)\ar[r]&0,}$$ where $u_{1}$ is induced by $f$ such that
$\im(u_{1})= \im(f(M\otimes f))$. We can form this diagram again with $(X,f)$ replaced by $(\im(f),u_{1})$ and get a left $R\ltimes_{n}M$-module $(\im(u_{1}),u_{2})$ such that $\im(u_{2})=\im(f(M\otimes f)(M^{\otimes_{R}2}\otimes f))$. The next step gives us a left $R\ltimes_{n}M$-module $(\im(u_{2}),u_{3})$ such that $\im(u_{3})= \im(f(M\otimes f)(M^{\otimes_{R}2}\otimes f)(M^{\otimes_{R}3}\otimes f))$. Since $f(M\otimes f)\cdots(M^{\otimes_{R}n}\otimes f)=0$, we have $u_{n}=0$. So we will eventually reach the above commutative diagram with the two extreme homomorphisms equal to zero by this process, i.e., there exist  exact sequences $$0\rightarrow (\im(f),u_{1})\rightarrow(X,f)\rightarrow\textbf{Z}(\coker(f))\rightarrow 0,$$ $$0\rightarrow(\im(u_{1}),u_{2})\rightarrow(\im(f),u_{1})\rightarrow\textbf{Z}(\coker(u_{1}))\rightarrow 0,$$$$\cdots\cdots$$ $$0\rightarrow\textbf{Z}(\im(u_{n-1}))\rightarrow(\im(u_{n-2}),u_{n-1})\rightarrow\textbf{Z}(\coker(u_{n-1}))\rightarrow 0.$$

This completes the proof.
\end{proof}
\begin{lem}\label{lem: 2.7}Let $(X,f)$ and $[Y,g]$ be left $R\ltimes_{n}M$-modules.
\begin{enumerate} \item  There is an exact sequence $M\otimes_{R}\coker(f_{i-1})\rightarrow \coker(f_{i})\rightarrow \coker(f)\rightarrow0$ for any $2\leq i\leq n+1$, where  $f_{j}=f(M\otimes f) \cdots(M^{\otimes_{R}j-1}\otimes f): M^{\otimes_{R}j}\otimes_{R} X\rightarrow X$.

Moreover, the sequence $0\rightarrow M\otimes_{R}\coker(f_{i-1})\rightarrow \coker(f_{i})\rightarrow \coker(f)\rightarrow 0$ is exact if and only if the sequence $M^{\otimes_{R}i}\otimes_{R} X\stackrel{M\otimes f_{i-1}}\longrightarrow  M\otimes_{R} X\stackrel{\pi_{i} f}\longrightarrow \coker(f_{i})$ is exact, where $\pi_{j}: X\rightarrow\coker(f_{j})$ is the canonical epimorphism.\item  There is an exact sequence  $0\rightarrow\ker(g)\longrightarrow \ker(g_{i})\longrightarrow \Hom_{R}(M,\ker(g_{i-1}))$ for any $2\leq i\leq n+1$, where  $g_{j}=\Hom_{R}^{j-1}(M,g)\cdots\Hom_{R}(M,g)g: Y\rightarrow \Hom_{R}^{j}(M,Y)$.

Moreover, the sequence $0\rightarrow\ker(g)\longrightarrow \ker(g_{i})\longrightarrow \Hom_{R}(M,\ker(g_{i-1}))\rightarrow 0$ is exact if and only if the sequence $\ker(g_{i})\stackrel{g\iota_{i}}\longrightarrow \Hom_{R}(M,Y)\stackrel{{\rm Hom}_{R}(M,g_{i-1})}\longrightarrow \Hom_{R}^{i}(M,Y)$ is exact, where $\iota_{j}: \ker(g_{j})\rightarrow Y$ is the inclusion.
\end{enumerate}
\end{lem}
\begin{proof}(1) The exact sequence $$M^{\otimes_{R}i-1}\otimes_{R} X\stackrel{f_{i-1}}\longrightarrow  X\stackrel{\pi_{i-1}}\longrightarrow \coker(f_{i-1})\rightarrow 0$$ induces the exact sequence $$M^{\otimes_{R}i}\otimes_{R} X\stackrel{M\otimes f_{i-1}}\longrightarrow  M\otimes_{R} X\stackrel{M\otimes \pi_{i-1}}\longrightarrow M\otimes_{R}\coker(f_{i-1})\rightarrow 0.$$ Since the composition $$M^{\otimes_{R}i}\otimes_{R} X\stackrel{M\otimes f_{i-1}}\longrightarrow  M\otimes_{R} X\stackrel{\pi_{i} f}\longrightarrow \coker(f_{i})$$ is 0, there is $\alpha_{i}: M\otimes_{R}\coker(f_{i-1})\rightarrow \coker(f_{i})$ such that $\alpha_{i}(M\otimes \pi_{i-1})=\pi_{i} f$. Define $\rho_{i}: \coker(f_{i})\longrightarrow \coker(f)$ by $\rho_{i}(\pi_{i}(x))=\pi_{1}(x)$. Note that $\im(\alpha_{i})=\im(\alpha_{i}(M\otimes \pi_{i-1}))=\pi_{i}(\im(f))$. So we get the exact sequence $$M\otimes_{R}\coker(f_{i-1})\stackrel{\alpha_{i}}\longrightarrow \coker(f_{i})\stackrel{\rho_{i}}\longrightarrow \coker(f)\rightarrow 0.$$

Moreover, the sequence $0\rightarrow M\otimes_{R}\coker(f_{i-1})\stackrel{\alpha_{i}}\longrightarrow \coker(f_{i})\stackrel{\rho_{i}}\longrightarrow \coker(f)\rightarrow 0$ is exact if and only if the sequence $M^{\otimes_{R}i}\otimes_{R} X\stackrel{M\otimes f_{i-1}}\longrightarrow  M\otimes_{R} X\stackrel{\pi_{i} f}\longrightarrow \coker(f_{i})$ is exact.

(2) The exact sequence $$0\rightarrow \ker(g_{i-1})\stackrel{\iota_{i-1}}\longrightarrow Y\stackrel{g_{i-1}}\longrightarrow \Hom_{R}^{i-1}(M,Y)$$ induces the exact sequence $$0\rightarrow \Hom_{R}(M,\ker(g_{i-1}))\stackrel{{\rm Hom}_{R}(M,\iota_{i-1})}\longrightarrow \Hom_{R}(M,Y)\stackrel{{\rm Hom}_{R}(M,g_{i-1})}\longrightarrow \Hom_{R}^{i}(M,Y).$$ Since the composition $$\ker(g_{i})\stackrel{g\iota_{i}}\longrightarrow \Hom_{R}(M,Y)\stackrel{{\rm Hom}_{R}(M,g_{i-1})}\longrightarrow \Hom_{R}^{i}(M,Y)$$ is 0, there is $\beta_{i}: \ker(g_{i})\rightarrow \Hom_{R}(M,\ker(g_{i-1}))$ such that $\Hom_{R}(M,\iota_{i-1})\beta_{i}=g\iota_{i}$. Note that $\ker(\beta_{i})=\ker(g)$. So we get the exact sequence $$0\rightarrow\ker(g)\longrightarrow \ker(g_{i})\stackrel{\beta_{i}}\longrightarrow \Hom_{R}(M,\ker(g_{i-1})).$$

Moreover, the sequence $0\rightarrow\ker(g)\longrightarrow \ker(g_{i})\stackrel{\beta_{i}}\longrightarrow \Hom_{R}(M,\ker(g_{i-1}))\rightarrow 0$ is exact if and only if the sequence $\ker(g_{i})\stackrel{g\iota_{i}}\longrightarrow \Hom_{R}(M,Y)\stackrel{{\rm Hom}_{R}(M,g_{i-1})}\longrightarrow \Hom_{R}^{i}(M,Y)$ is exact.
\end{proof}
Now we give characterizations of projective, injective and flat modules over $R\ltimes_{n}M$, which generalize \cite[Corollary 1.6 and Proposition 1.14]{FGR} and \cite[Proposition 1]{PR}. We also point out that the method of proof here  is different from \cite[Corollaries 3.3, 3.6 and 3.10]{BBG}.
\begin{thm}\label{thm: 2.8}The following conditions are equivalent for a left $R\ltimes_{n}M$-module $(X,f)$:
\begin{enumerate} \item $(X,f)$ is a  projective left $R\ltimes_{n}M$-module.\item The sequence $M^{\otimes_{R}n+1}\otimes_{R} X\stackrel{(M\otimes f) \cdots(M^{\otimes_{R}n}\otimes f)}\longrightarrow  M\otimes_{R} X\stackrel{f}\longrightarrow X$ is exact and $\coker(f)$ is a projective  left $R$-module.\item $(X,f)\cong \textbf{T}(P)$ with $P$ a projective  left $R$-module.
\end{enumerate}
\end{thm}
\begin{proof}(1) $\Rightarrow (2)$ There is a split epimorphism $\theta: \textbf{T}(R)^{(I)}\rightarrow (X,f)$. Therefore we get the split epimorphism $\textbf{C}(\theta): \textbf{C}(\textbf{T}(R)^{(I)})\rightarrow \textbf{C}(X,f)$. Since $\textbf{C}(\textbf{T}(R)^{(I)})\cong R^{(I)}$, $\coker(f)=\textbf{C}(X,f)$ is a projective  left $R$-module.

Consider the following commutative diagram with exact rows: $$\xymatrix{&M^{\otimes_{R}n+1}\otimes_{R}(\coprod_{i=0}^{n}M^{\otimes_{R}i})^{(I)}\ar[d]_{(M\otimes \mu^{(I)}) \cdots(M^{\otimes_{R}n}\otimes \mu^{(I)})}\ar[rrr]^{M^{\otimes_{R}n+1}\otimes\theta}&&&M^{\otimes_{R}n+1}\otimes_{R} X\ar[r]\ar[d]^{(M\otimes f) \cdots(M^{\otimes_{R}n}\otimes f)}&0\\&M\otimes_{R}(\coprod_{i=0}^{n}M^{\otimes_{R}i})^{(I)}\ar[rrr]^{M\otimes\theta}\ar[d]_{\mu^{(I)}}&&&M\otimes_{R} X\ar[r]\ar[d]^{f}&0\\&(\coprod_{i=0}^{n}M^{\otimes_{R}i})^{(I)}\ar[rrr]^{\theta}&&&X\ar[r]&0.}$$ Since $\theta$ is a split epimorphism and the left column is exact, diagram chasing shows that the right column is also exact.

(2) $\Rightarrow (3)$ Let $f_{i}=f(M\otimes f) \cdots(M^{\otimes_{R}i-1}\otimes f): M^{\otimes_{R}i}\otimes_{R} X\longrightarrow X$ and $\pi_{i}: X\longrightarrow\coker(f_{i})$ be the canonical epimorphism. Since the sequence $M^{\otimes_{R}n+1}\otimes_{R} X\stackrel{M\otimes f_{n}}\longrightarrow  M\otimes_{R} X\stackrel{f}\longrightarrow X$ is exact, we get the exact sequence $$0\longrightarrow M\otimes_{R}\coker(f_{n})\stackrel{\alpha_{n+1}}\longrightarrow X\stackrel{\pi_{1}}\longrightarrow \coker(f)\longrightarrow 0$$  such that $\alpha_{n+1}(M\otimes \pi_{n})=f$ by Lemma \ref{lem: 2.7}(1). Since the sequence $$M^{\otimes_{R}n}\otimes_{R} X\stackrel{M\otimes f_{n-1}}\longrightarrow  M\otimes_{R} X\stackrel{\pi_{n}f}\longrightarrow \coker(f_{n})$$  is exact, we get the exact sequence $$0\longrightarrow M\otimes_{R}\coker(f_{n-1})\stackrel{\alpha_{n}}\longrightarrow \coker(f_{n})\stackrel{\rho_{n}}\longrightarrow \coker(f)\longrightarrow 0$$  such that $\alpha_{n}(M\otimes \pi_{n-1})=\pi_{n} f$ by Lemma \ref{lem: 2.7}(1). Continuing this process, we get the exact sequences   $$0\longrightarrow M\otimes_{R}\coker(f_{i})\stackrel{\alpha_{i+1}}\longrightarrow \coker(f_{i+1})\stackrel{\rho_{i+1}}\longrightarrow \coker(f)\longrightarrow 0,$$ $$\cdots\cdots$$ $$0\longrightarrow M\otimes_{R}\coker(f)\stackrel{\alpha_{2}}\longrightarrow \coker(f_{2})\stackrel{\rho_{2}}\longrightarrow \coker(f)\longrightarrow 0$$  such that $\alpha_{i}(M\otimes \pi_{i-1})=\pi_{i} f$. Since $\coker(f)$ is a projective  left $R$-module, the sequence $$0\longrightarrow M\otimes_{R}\coker(f_{n})\stackrel{\alpha_{n+1}}\longrightarrow X\stackrel{\pi_{1}}\longrightarrow \coker(f)\longrightarrow 0$$ is split. So there is $\iota:  \coker(f)\longrightarrow X$ such that $\pi_{1}\iota=1$.

Define $\gamma: \coprod_{i=0}^{n}(M^{\otimes_{R}i}\otimes_{R}\coker(f))\rightarrow X$ by $\gamma(x_{0}, x_{1},\cdots, x_{n})=\iota(x_{0})+\alpha_{n+1}(M\otimes\pi_{n}\iota)(x_{1})+\alpha_{n+1}(M\otimes\alpha_{n})(M\otimes_{R} M\otimes\pi_{n-1}\iota)(x_{2})+\cdots+\alpha_{n+1}(M\otimes\alpha_{n})\cdots
(M^{\otimes_{R}n-1}\otimes\pi_{2}\iota)(x_{n-1})+\alpha_{n+1}(M\otimes\alpha_{n})\cdots(M^{\otimes_{R}n-1}\otimes\alpha_{2})(x_{n})$. It is easy to check that $\gamma$ is a left $R$-isomorphism such that $\gamma\mu=f(M\otimes\gamma)$. So $(X,f)\cong \textbf{T}(\coker(f))$.

(3) $\Rightarrow (1)$ is obvious.
\end{proof}
\begin{thm}\label{thm: 2.9}The following conditions are equivalent for a left $R\ltimes_{n}M$-module $[Y,g]$:
\begin{enumerate} \item $[Y,g]$ is an injective left $R\ltimes_{n}M$-module.\item The sequence $Y\stackrel{g}\longrightarrow \Hom_{R}(M,Y)\stackrel{{\rm Hom}_{R}^{n}(M,g)\cdots{\rm Hom}_{R}(M,g)}\longrightarrow \Hom^{n+1}_{R}(M,Y)$ is exact and $\ker(g)$ is an injective  left $R$-module.\item $[Y,g]\cong \textbf{H}(E)$ with $E$ an injective  left $R$-module.
\end{enumerate}
\end{thm}
\begin{proof}(1) $\Rightarrow (2)$ There is a split monomorphism $\beta: [Y,g]\rightarrow(\textbf{T}(R)^{+})^{I}$. Then we get the split monomorphism $\textbf{K}(\beta): \textbf{K}[Y,g]\rightarrow\textbf{K}((\textbf{T}(R)^{+})^{I})$. Since $\textbf{K}((\textbf{T}(R)^{+})^{I})\cong (R^{+})^{I}$ is injective, $\ker(g)=\textbf{K}[Y,g]$ is an injective  left $R$-module.

Consider the following commutative diagram with exact rows: $$\xymatrix{0\ar[r]&Y\ar[d]\ar[rrr]^{\beta}&&&(\coprod_{i=0}^{n}\Hom^{i}_{R}(M,R^{+}))^{I}\ar[d] \\0\ar[r]&\Hom_{R}(M,Y)\ar[d]
\ar[rrr]^{{\rm Hom}_{R}(M,\beta)}&&&(\coprod_{i=0}^{n}\Hom^{i+1}_{R}(M,R^{+}))^{I}\ar[d]\\0\ar[r]&\Hom^{n+1}_{R}(M,Y)\ar[rrr]^{{\rm Hom}_{R}^{n+1}(M,\beta)}&&&(\coprod_{i=0}^{n}\Hom^{i+n+1}_{R}(M,R^{+}))^{I}.}$$ Since $\beta$ is a split monomorphism and the right column is exact, diagram chasing shows that the left column is also exact.

(2) $\Rightarrow (3)$ Let $g_{i}=\Hom_{R}^{i-1}(M,g)\cdots\Hom_{R}(M,g)g: Y\rightarrow \Hom_{R}^{i}(M,Y)$ and $\iota_{i}: \ker(g_{i})\rightarrow Y$ be the inclusion. Since the sequence $Y\stackrel{g}\longrightarrow \Hom_{R}(M,Y)\stackrel{{\rm Hom}_{R}^{n}(M,g)\cdots{\rm Hom}_{R}(M,g)}\longrightarrow \Hom^{n+1}_{R}(M,Y)$ is exact, we get the exact sequence $$0\rightarrow\ker(g)\longrightarrow Y\longrightarrow \Hom_{R}(M,\ker(g_{n}))\rightarrow 0$$ by Lemma \ref{lem: 2.7}(2). Since the sequence $\ker(g_{n})\stackrel{g\iota_{n}}\longrightarrow \Hom_{R}(M,Y)\stackrel{{\rm Hom}_{R}^{n-1}(M,g)\cdots{\rm Hom}_{R}(M,g)}\longrightarrow \Hom^{n}_{R}(M,Y)$ is exact,  we get the exact sequence $$0\rightarrow\ker(g)\longrightarrow \ker(g_{n})\longrightarrow \Hom_{R}(M,\ker(g_{n-1}))\rightarrow 0$$ by Lemma \ref{lem: 2.7}(2). Continuing this process, we get the exact sequences $$0\rightarrow\ker(g)\longrightarrow \ker(g_{i+1})\longrightarrow \Hom_{R}(M,\ker(g_{i}))\rightarrow 0,$$ $$\cdots\cdots$$ $$0\rightarrow\ker(g)\longrightarrow \ker(g_{2})\longrightarrow \Hom_{R}(M,\ker(g))\rightarrow 0.$$  Since $\ker(g)$ is an injective  left $R$-module, all exact sequences $$0\rightarrow\ker(g)\longrightarrow \ker(g_{i+1})\longrightarrow \Hom_{R}(M,\ker(g_{i}))\rightarrow 0$$ are split. Thus $Y\cong\coprod_{i=0}^{n}\Hom_{R}^{i}(M,\ker(g))$. It is easy to verify that $[Y,g]\cong \textbf{H}(\ker(g))$.

(3) $\Rightarrow (1)$ is clear.
\end{proof}
\begin{thm}\label{thm: 2.10} $(X,f)$ is a flat left $R\ltimes_{n}M$-module if and only if the sequence $M^{\otimes_{R}n+1}\otimes_{R} X\stackrel{(M\otimes f) \cdots(M^{\otimes_{R}n}\otimes f)}\longrightarrow  M\otimes_{R} X\stackrel{f}\longrightarrow X$ is exact and $\coker(f)$ is a flat left $R$-module.
\end{thm}
\begin{proof}We write $\omega: (M\otimes_{R} X)^{+}\rightarrow\Hom_{R}(M, X^{+})$ to be the natural isomorphism. Then $(X,f)$ is a flat left $R\ltimes_{n}M$-module  if and only if  $[X^{+},\omega f^{+}]$ is an injective right $R\ltimes_{n}M$-module if and only if the sequence $X^{+}\stackrel{\omega f^{+}}\longrightarrow \Hom_{R}(M,X^{+})\stackrel{{\rm Hom}_{R}^{n}(M,\omega f^{+})\cdots{\rm Hom}_{R}(M,\omega f^{+})}\longrightarrow \Hom^{n+1}_{R}(M,X^{+})$ is exact and $\ker(\omega f^{+})$ is an injective right $R$-module by Theorem \ref{thm: 2.9} if and only if the sequence $M^{\otimes_{R}n+1}\otimes_{R} X\stackrel{(M\otimes f) \cdots(M^{\otimes_{R}n}\otimes f)}\longrightarrow  M\otimes_{R} X\stackrel{f}\longrightarrow X$ is exact and $\coker(f)$ is a flat left $R$-module.
\end{proof}
\begin{cor}\label{cor: 2.11}Let $X$ be a left $R$-module. Then
\begin{enumerate}\item $X$ is a projective left $R$-module if and only if $\textbf{T}(X)$ is a projective  left $R\ltimes_{n}M$-module.
\item $X$ is an injective left $R$-module if and only if $\textbf{H}(X)$ is an injective left $R\ltimes_{n}M$-module.\item  $X$ is a flat left $R$-module if and only if $\textbf{T}(X)$ is a flat left $R\ltimes_{n}M$-module.
\end{enumerate}
\end{cor}
\begin{proof}It is an immediate consequence of Theorems \ref{thm: 2.8}, \ref{thm: 2.9} and \ref{thm: 2.10}.
\end{proof}
Next we describe finitely generated and finitely presented modules over $R\ltimes_{n}M$.
\begin{thm}\label{thm: 2.12}Let $(X, f)$ be a left $R\ltimes_{n}M$-module.
\begin{enumerate}\item  $(X, f)$ is a finitely generated left $R\ltimes_{n}M$-module if and only if  $\coker(f)$ is a finitely generated left $R$-module.\item  If $(X, f)$ is a finitely presented left $R\ltimes_{n}M$-module, then $\coker(f)$ is a finitely presented left $R$-module. The converse holds if $\ker(f)$ is a finitely generated left $R$-module.
\end{enumerate}
\end{thm}
\begin{proof}(1) $``\Rightarrow"$ There is an epimorphism $\theta: \textbf{T}(R)^{m}\rightarrow (X,f)$ with  $m\in\mathbb{N}$. Therefore we get the  epimorphism $\textbf{C}(\theta): \textbf{C}(\textbf{T}(R)^{m})\rightarrow \textbf{C}(X,f)$. Since $\textbf{C}(\textbf{T}(R)^{m})\cong R^{m}$, $\coker(f)=\textbf{C}(X,f)$ is a  finitely generated left $R$-module.

$``\Leftarrow"$  Let $f_{i}=f(M\otimes f) \cdots(M^{\otimes_{R}i-1}\otimes f): M^{\otimes_{R}i}\otimes_{R} X\rightarrow X$ and $\pi_{i}: X\rightarrow\coker(f_{i})$ be the canonical epimorphism. By Lemma \ref{lem: 2.7}(1), there exist exact sequences $$M\otimes_{R}\coker(f)\stackrel{\alpha_{2}}\longrightarrow \coker(f_{2})\longrightarrow \coker(f)\rightarrow 0,$$  $$0\rightarrow M\otimes_{R}\coker(f_{2})\stackrel{\alpha_{3}}\longrightarrow \coker(f_{3})\longrightarrow \coker(f)\rightarrow 0,$$ $$\cdots\cdots$$ $$0\rightarrow M\otimes_{R}\coker(f_{n})\stackrel{\alpha_{n+1}}\longrightarrow X\longrightarrow \coker(f)\rightarrow 0$$ such that $\alpha_{i}(M\otimes \pi_{i-1})=\pi_{i} f$. Let $\coker(f)=\sum\limits_{j=1}^{l}R\pi_{1}(x_{j})$ with $x_{j}\in X$ and $l\in \mathbb{N}$.  For any $x\in X$, $\pi_{1}(x)=\sum\limits_{j=1}^{l}r_{j}\pi_{1}(x_{j})=\pi_{1}(\sum\limits_{j=1}^{l}r_{j}x_{j})$
with $r_{j}\in R$.  Hence $x-\sum\limits_{j=1}^{l}r_{j}x_{j}\in \ker(\pi_{1})=\im(f)$. So $\pi_{2}(x-\sum\limits_{j=1}^{l}r_{j}x_{j})\in\im(\alpha_{2}).$ Thus there exists some $y\in M\otimes_{R}\coker(f)$ such that $\pi_{2}(x-\sum\limits_{j=1}^{l}r_{j}x_{j})=\alpha_{2}(y).$
Let $y=\sum\limits_{t=1}^{k}m_{t}\otimes\sum\limits_{j=1}^{l}s_{tj}\pi_{1}(x_{j})$ with $s_{tj}\in R$. Then $$\pi_{2}(x)=\pi_{2}(\sum_{j=1}^{l}r_{j}x_{j})+\alpha_{2}(y)=\pi_{2}(\sum_{j=1}^{l}r_{j}x_{j})+\alpha_{2}(\sum\limits_{t=1}^{k}m_{t}\otimes\sum\limits_{j=1}^{l}s_{tj}\pi_{1}(x_{j}))$$
$$=\pi_{2}(\sum_{j=1}^{l}r_{j}x_{j})+\pi_{2}(\sum\limits_{t=1}^{k}\sum_{j=1}^{l} f(m_{t}s_{tj}\otimes x_{j}))$$$$=\pi_{2}(\sum_{j=1}^{l}(r_{j},0,0,\cdots)x_{j}+\sum_{j=1}^{l}\sum\limits_{t=1}^{k}(0,m_{t}s_{tj},0,0,\cdots)x_{j})\in\pi_{2}(\sum_{j=1}^{l}(R\ltimes_{n}M)x_{j}).$$
Similarly, $\pi_{3}(x)\in\pi_{3}(\sum\limits_{j=1}^{l}(R\ltimes_{n}M)x_{j}),$ $\cdots$, $\pi_{n+1}(x)\in\pi_{n+1}(\sum\limits_{j=1}^{l}(R\ltimes_{n}M)x_{j}),$ i.e., $x\in\sum\limits_{j=1}^{l}(R\ltimes_{n}M)x_{j}.$
So $(X, f)=\sum\limits_{j=1}^{l}(R\ltimes_{n}M)x_{j}$ is a finitely generated left $R\ltimes_{n}M$-module.

(2) If $(X, f)$ is a finitely presented left $R\ltimes_{n}M$-module, then there is an exact sequence $ \textbf{T}(R)^{k}\rightarrow \textbf{T}(R)^{m}\rightarrow (X, f)\rightarrow 0$ in $R\ltimes_{n}M$-Mod with  $m,k\in\mathbb{N}$, which induces the exact sequence $ R^{k}\rightarrow R^{m}\rightarrow \coker(f)\rightarrow 0$. So $\coker(f)$ is a finitely presented left $R$-module.

Conversely, if $\coker(f)$ is a finitely presented left $R$-module and $\ker(f)$ is a finitely generated left $R$-module, then  $(X, f)$ is a finitely generated left $R\ltimes_{n}M$-module by (1). So there is an exact sequence $0\rightarrow (A,h)\rightarrow \textbf{T}(R)^{k}\rightarrow (X, f)\rightarrow 0$ with $k\in\mathbb{N}$.  Consider the commutative diagram with exact rows: $$\xymatrix{&M\otimes_{R}A\ar[r]\ar[d]_{h}&M\otimes_{R}(\coprod_{i=0}^{n}M^{\otimes_{R}i})^{k}\ar[d]_{\mu^{k}}\ar[r]&M\otimes_{R}X\ar[d]_{f}\ar[r]&0\\
0\ar[r]&A\ar[r]&(\coprod_{i=0}^{n}M^{\otimes_{R}i})^{k}\ar[r]&X\ar[r]&0.}$$ By the snake lemma, we get the exact sequence $$\ker(f)\rightarrow\coker(h)\rightarrow R^{k}\rightarrow\coker(f)\rightarrow 0.$$ Thus $\coker(h)$ is a finitely generated left $R$-module.  By (1), $(A,h)$ is a finitely generated left $R\ltimes_{n}M$-module. So $(X, f)$ is a finitely presented left $R\ltimes_{n}M$-module.
\end{proof}
\begin{cor}\label{cor: 2.13}Let $X$ be a left $R$-module. Then
\begin{enumerate}\item  $X$ is a finitely generated left $R$-module if and only if $\textbf{T}(X)$ is a finitely generated left $R\ltimes_{n}M$-module.
\item $X$ is a finitely presented left $R$-module if and only if $\textbf{T}(X)$ is a finitely presented left $R\ltimes_{n}M$-module.
\end{enumerate}
\end{cor}
\begin{proof}It is an immediate consequence of Theorem \ref{thm: 2.12}.
\end{proof}
\bigskip
\section {\bf Properties passing over to  $n$-trivial extensions of rings}
\bigskip
Using the foregoing results, we first characterize when the ring $R\ltimes_{n}M$ is left perfect (Noetherian, Artinian, coherent, hereditary), which generalizes \cite[Proposition 1.15]{FGR} and \cite[Propositions 1.5.1(b,c) and 2.3.3]{R}.
\begin{thm}\label{thm: 3.1}Let $R\ltimes_{n} M$ be an $n$-trivial extension. Then
\begin{enumerate}\item $R\ltimes_{n}M$ is a left  perfect ring if and only if $R$ is a  left  perfect  ring.\item  $R\ltimes_{n}M$ is a left Noetherian ring if and only if $R$ is a left Noetherian ring and $M$ is a finitely generated left $R$-module.\item  $R\ltimes_{n}M$ is a left Artinian ring if and only if $R$ is a left Artinian ring and $M$ is a finitely generated left $R$-module.\item $R\ltimes_{n}M$ is a left coherent ring if and only if $R$ is a left coherent ring, the sequence $M^{\otimes_{R}n+1}\otimes_{R} (\coprod_{i=0}^{n}M^{\otimes_{R}i})^{I}\stackrel{(M\otimes h) \cdots(M^{\otimes_{R}n}\otimes h)}\longrightarrow  M\otimes_{R} (\coprod_{i=0}^{n}M^{\otimes_{R}i})^{I}\stackrel{h}\longrightarrow (\coprod_{i=0}^{n}M^{\otimes_{R}i})^{I}$  is exact for any index set $I$ and $\coker(\gamma)$ is a flat right $R$-module, where $\gamma: M\otimes_{R} (\coprod_{i=0}^{n-1}M^{\otimes_{R}i})^{I}\rightarrow (\coprod_{i=1}^{n}M^{\otimes_{R}i})^{I}$ is defined in an obvious way.\item $R\ltimes_{n} M$ is a left hereditary ring if and only if  $R$ is a left hereditary ring, $M^{\otimes_{R}n+1}=0$, $M_{R}$ is flat, $M\otimes_{R}X$ is a projective left $R$-module for every left $R$-module $X$.
\end{enumerate}
\end{thm}
\begin{proof}(1) $``\Rightarrow"$  Let $Y$ be a flat left $R$-module. Then $\textbf{T}(Y)$ is a flat left $R\ltimes_{n}M$-module by Corollary \ref{cor: 2.11}(3) and so is a projective  left $R\ltimes_{n}M$-module. Thus $Y$ is a projective  left $R$-module by Corollary \ref{cor: 2.11}(1). Hence $R$ is a  left  perfect  ring.

$``\Leftarrow"$ Let $(X,f)$ be a  flat  left $R\ltimes_{n}M$-module. By Theorem \ref{thm: 2.10}, the sequence $M^{\otimes_{R}n+1}\otimes_{R} X\stackrel{(M\otimes f) \cdots(M^{\otimes_{R}n}\otimes f)}\longrightarrow  M\otimes_{R} X\stackrel{f}\longrightarrow X$ is exact and $\coker(f)$ is a  flat  left $R$-module and so is a  projective left $R$-module. Thus $(X,f)$ is a projective left $R\ltimes_{n}M$-module by Theorem  \ref{thm: 2.8}. Hence $R\ltimes_{n}M$ is a left perfect ring.

(2) $``\Rightarrow"$  Let $Y$ be a finitely generated left $R$-module. Then $\textbf{T}(Y)$ is a finitely generated left $R\ltimes_{n}M$-module by Corollary \ref{cor: 2.13}(1) and so is finitely presented. Hence $Y$ is finitely presented  by Corollary \ref{cor: 2.13}(2). Thus $R$ is a left Noetherian ring. Also $\coprod_{i=1}^{n}M^{\otimes_{R}i}$ is a left ideal of $R\ltimes_{n}M$ and so is finitely generated. Thus  $M$ is a finitely generated left $R$-module.

$``\Leftarrow"$ Let $(X,f)$ be a finitely generated left $R\ltimes_{n}M$-module. Then $\coker(f)$ is a finitely generated left $R$-module  by  Theorem  \ref{thm: 2.12}(1)  and so is finitely presented. Since  $X$ is a finitely generated left $R$-module, $\ker(f)$ is a finitely generated left $R$-module. So $(X,\alpha)$ is finitely presented  by Theorem  \ref{thm: 2.12}(2). Hence $R\ltimes_{n}M$ is a left Noetherian ring.

(3) follows from (1) and (2).

(4) $R\ltimes_{n}M$ is a left coherent ring if and only if $\textbf{T}(R)^{I}\cong((\coprod_{i=0}^{n}M^{\otimes_{R}i})^{I},h)$  is a flat right $R\ltimes_{n}M$-module for any index set $I$ by \cite[Theorem 2.1]{C} if and only if  $\coker(h)\cong R^{I}\oplus \coker(\gamma)$ is a flat right $R$-module and the sequence $M^{\otimes_{R}n+1}\otimes_{R} (\coprod_{i=0}^{n}M^{\otimes_{R}i})^{I}\stackrel{(M\otimes h) \cdots(M^{\otimes_{R}n}\otimes h)}\longrightarrow  M\otimes_{R} (\coprod_{i=0}^{n}M^{\otimes_{R}i})^{I}\stackrel{h}\longrightarrow (\coprod_{i=0}^{n}M^{\otimes_{R}i})^{I}$ is exact  by Theorem \ref{thm: 2.10}  if and only if $R$ is a left coherent ring, $\coker(\gamma)$ is a flat right $R$-module and the sequence $M^{\otimes_{R}n+1}\otimes_{R} (\coprod_{i=0}^{n}M^{\otimes_{R}i})^{I}\stackrel{(M\otimes h) \cdots(M^{\otimes_{R}n}\otimes h)}\longrightarrow  M\otimes_{R} (\coprod_{i=0}^{n}M^{\otimes_{R}i})^{I}\stackrel{h}\longrightarrow (\coprod_{i=0}^{n}M^{\otimes_{R}i})^{I}$  is exact.

(5) $``\Rightarrow"$ Note that $\coprod_{i=1}^{n}M^{\otimes_{R}i}$ is a left ideal of $R\ltimes_{n} M$ and so is projective. By Theorem \ref{thm: 2.8}, the sequence $M^{\otimes_{R}n+1}\otimes_{R} \coprod_{i=1}^{n}M^{\otimes_{R}i}\rightarrow M\otimes_{R}\coprod_{i=1}^{n}M^{\otimes_{R}i}\rightarrow \coprod_{i=1}^{n}M^{\otimes_{R}i}$ is exact and $_{R}M$ is projective. Thus $M^{\otimes_{R}n+1}=0$.

Let $I$ be any left ideal of $R$. Then $\coprod_{i=0}^{n}M^{\otimes_{R}i}I$ is a left ideal of $R\ltimes_{n} M$ and so is projective. By Theorem \ref{thm: 2.8}, the sequence $0=M^{\otimes_{R}n+1}\otimes_{R} \coprod_{i=0}^{n}M^{\otimes_{R}i}I\rightarrow M\otimes_{R}\coprod_{i=0}^{n}M^{\otimes_{R}i}I\rightarrow \coprod_{i=0}^{n}M^{\otimes_{R}i}I$ is exact and $I$ is projective. Thus $M \otimes_{R}I\rightarrow MI$ is a monomorphism. So $M_{R}$ is flat and $R$ is a left hereditary ring.

For every left $R$-module $X$, there is an exact sequence $0\rightarrow (K,h)\rightarrow \textbf{T}(R)^{(I)}\rightarrow\textbf{Z}(X)\rightarrow 0$. By the snake lemma, we get the exact sequence $0\rightarrow M\otimes_{R}X\rightarrow \coker(h)$. Since $(K,h)$ is a projective left $R\ltimes_{n} M$-module, $\coker(h)$ is a projective left $R$-module. Thus  $M\otimes_{R}X$ is projective left $R$-module.

$``\Leftarrow"$ Let $(X,f)$ be a left ideal of $R\ltimes_{n} M$. Then  there is an exact sequence $0\rightarrow (X,f)\rightarrow \textbf{T}(R)\rightarrow (L,g)\rightarrow 0$. By the snake lemma, we get the exact sequence $$0\rightarrow \ker(f)\rightarrow 0\rightarrow \ker(g)\rightarrow \coker(f)\rightarrow R\rightarrow \coker(g)\rightarrow 0.$$ Thus $f$ is a monomorphism. Since $M\otimes_{R} L$ is projective and $\ker(g)$ is a submodule of $M\otimes_{R} L$, we have $\ker(g)$  is projective. So  $\coker(f)$ is projective. Hence $(X,f)$ is projective by Theorem \ref{thm: 2.8}. Thus $R\ltimes_{n} M$ is a left hereditary ring.
\end{proof}
Reiten characterized when a trivial extension $R\ltimes M$ is a left self-injective  ring \cite[Theorem 1.4.1]{R}. Next we give a sufficient and necessary  condition  for $R\ltimes_{n} M$ to be left self-injective in a different way.

Let $A_{R}$ and $_{R}B$ be modules. Write $Ann_{B}(A)=\{b\in B: a\otimes b=0$ in $A\otimes_{R}B$ for all $a\in A\}$.

Define $\tau_{1}:\ _{R}R\rightarrow \Hom_{R}(M,M)$  by $\tau_{1}(r)(m)=mr$ and  $\tau_{i}: M^{\otimes_{R} i-1}\rightarrow\Hom_{R}(M,M^{\otimes_{R} i})$ by $\tau_{i}(m)(x)=x\otimes m$ for $i\geq 2$.
\begin{thm}\label{thm: 3.2}$R\ltimes_{n} M$ is a left self-injective  ring if and only if the following conditions are satisfied:
\begin{enumerate}\item $M$ and $\coprod_{i=0}^{n-1}Ann_{M^{\otimes_{R}i}}(M)$ are injective left $R$-modules,\item $\Hom_{R}(M, Ann_{R}(M^{\otimes_{R} n}))=0$, \item $\tau_{i}: M^{\otimes_{R} i-1}\rightarrow\Hom_{R}(M,M^{\otimes_{R} i})$  is an epimorphism  for any $1\leq i\leq n$.
\end{enumerate}
\end{thm}
\begin{proof}By Theorem \ref{thm: 2.9}, $R\ltimes_{n} M$ is a left self-injective  ring if and only if $[\coprod_{i=0}^{n}M^{\otimes_{R}i},\xi]$ is an injective left $R\ltimes_{n} M$-module if and only if $\ker(\xi)$ is an  injective left $R$-module and the sequence of left $R$-modules $$\coprod_{i=0}^{n}M^{\otimes_{R}i}\stackrel{\xi}\longrightarrow\coprod_{i=0}^{n}\Hom_{R}(M,M^{\otimes_{R}i})\stackrel{{\rm Hom}_{R}^{n}(M,\xi)\cdots{\rm Hom}_{R}(M,\xi)}\longrightarrow\coprod_{i=0}^{n}\Hom^{n+1}_{R}(M,M^{\otimes_{R}i})$$ is exact, where $\xi$ is defined by $$\xi(m_{0},m_{1}, ..., m_{n})=(0,\gamma_{1}, ..., \gamma_{n}), \gamma_{1}(x)=xm_{0},\gamma_{2}(x)=x\otimes m_{1},\cdots,\gamma_{n}(x)=x\otimes m_{n-1}.$$ Then it is easy to check that $\Hom_{R}^{n}(M,\xi)\cdots\Hom_{R}(M,\xi)(g_{0}, ..., g_{n})=(0,0, ...,0,\delta_{n}),$ where $\delta_{n}(m_{1})(m_{2})...( m_{n+1})=(m_{n+1}\otimes m_{n}\otimes\cdots\otimes m_{2})g_{0}(m_{1}).$
Note that $$\ker(\xi)=\coprod_{i=0}^{n-1}Ann_{M^{\otimes_{R}i}}(M)\oplus M,$$ {\Small$$ \im(\xi)=\{\gamma_{1}\in\Hom_{R}(M,M):\gamma_{1}(x)=xm_{0}, m_{0}\in R\}\coprod_{i=2}^{n}\{\gamma_{i}\in\Hom_{R}(M,M^{\otimes_{R}i}):\gamma_{i}(x)=x\otimes m_{i}, m_{i}\in M^{\otimes_{R}i-1}\},$$} {\small $$
\ker(\Hom_{R}^{n}(M,\xi)\cdots\Hom_{R}(M,\xi))=\{g_{0}\in\Hom_{R}(M,R): M^{\otimes_{R} n} \im(g_{0})=0\} \coprod_{i=1}^{n}\Hom_{R}(M,M^{\otimes_{R}i}).$$}
Therefore $R\ltimes_{n} M$ is a left self-injective  ring if and only if $\coprod_{i=0}^{n-1}Ann_{M^{\otimes_{R}i}}(M)\oplus M$ is an injective left $R$-module, $\{g_{0}\in\Hom_{R}(M,R): M^{\otimes n} \im(g_{0})=0\}=0$, $\{\gamma_{1}\in\Hom_{R}(M,M):\gamma_{1}(x)=xm_{0}, m_{0}\in R\}=\Hom_{R}(M,M)$ and $\coprod_{i=2}^{n}\{\gamma_{i}\in\Hom_{R}(M,M^{\otimes_{R}i}):\gamma_{i}(x)=x\otimes m_{i}, m_{i}\in M^{\otimes_{R}i-1}\}=\coprod_{i=2}^{n}\Hom_{R}(M,M^{\otimes_{R}i})$ if and only if $M$ and $\coprod_{i=0}^{n-1}Ann_{M^{\otimes_{R}i}}(M)$ are injective left $R$-modules, $\Hom_{R}(M, Ann_{R}(M^{\otimes_{R} n}))=0$ and $\tau_{i}: M^{\otimes_{R} i-1}\rightarrow\Hom_{R}(M,M^{\otimes_{R} i})$  is an epimorphism  for any $1\leq i\leq n$.
\end{proof}
\begin{lem}\label{lem: 3.3}The following conditions are equivalent for a left $R\ltimes_{n}M$-module $(X,f)$:
\begin{enumerate}\item $(X, f)$ is a simple left $R\ltimes_{n}M$-module.\item $X$ is a simple left $R$-module and $ f=0$. \item $X$ is a simple left $R$-module.
\end{enumerate}
\end{lem}
\begin{proof}(1) $\Rightarrow$ (2)  By Lemma \ref{lem: 2.6}(3), there is an epimorphism $(X, f)\rightarrow  \textbf{Z}(\coker(f))$. So $\im(f)=X$ or $\im(f)=0$. If $\im(f)=X$, then $ f$ is an epimorphism. Since $f(M\otimes f)\cdots(M^{\otimes_{R}n}\otimes f)=0$,  $X=0$, a contradiction. Thus $ f= 0$.

Let $Y$ be a submodule of $X$. Then $(Y,0)$ is a submodule of $(X,0)$. So $Y=X$ or $Y=0$.  Thus $X$ is a simple  left $R$-module.

(2) $\Rightarrow$ (3) is trivial.

(3) $\Rightarrow$ (1) Let $(B,g)$ be a submodule of $(X, f)$. Then  $B$ is a submodule of $X$ and so $B=0$ or $B=X$. Hence $(B,g)=0$ or $(B,g)=(X, f)$. Thus $(X, f)$ is a simple left $R\ltimes_{n}M$-module.
\end{proof}
Recall that $R$ is a \emph{left $SF$ ring} if every simple left $R$-module is flat.
$R$ is called a \emph{left $V$-ring} if every simple left $R$-module is injective.
$R$ is called a \emph{left Kasch ring} \cite{L} if every simple left $R$-module can be embedded in $_{R}R$.
\begin{thm}\label{thm: 3.4}Let $R\ltimes_{n}M$ be an $n$-trivial  extension.
\begin{enumerate}\item $R\ltimes_{n} M$ is a semisimple Artinian ring if and only if  $R$ is a semisimple Artinian ring and $M\otimes_{R}X=0$ for any simple left $R$-module $X$.\item $R\ltimes_{n}M$ is a left $SF$ ring if and only if $R$ is a left $SF$ ring and $M\otimes_{R}X=0$ for any simple left $R$-module $X$.\item $R\ltimes_{n}M$ is a left $V$-ring if and only if $R$ is a left $V$-ring and $\Hom_{R}(M,X)=0$ for any simple left $R$-module $X$.\item  $R\ltimes_{n}M$ is a left Kasch ring  if and only if
$\Hom_{R}(X, \coprod_{i=0}^{n-1}Ann_{M^{\otimes_{R}i}}(M)\oplus M)\neq 0$  for any simple left $R$-module $X$.
\end{enumerate}
\end{thm}
\begin{proof}(1) $R\ltimes_{n} M$ is a semisimple Artinian ring if and only if any simple  left $R\ltimes_{n} M$-module is projective  if and only if any left $R\ltimes_{n} M$-module $\textbf{Z}(X)$ with $X$ a simple left $R$-module is projective by Lemma \ref{lem: 3.3} if and only if any simple  left $R$-module $X$ is projective and $M\otimes_{R}X=0$ by Theorem \ref{thm: 2.8} if and only if  $R$ is a semisimple Artinian ring  and $M\otimes_{R}X=0$ for any simple left $R$-module $X$.

(2) $R\ltimes_{n} M$ is a left $SF$ ring  if and only if any left $R\ltimes_{n} M$-module $\textbf{Z}(X)$ with $X$  a simple left $R$-module is flat by Lemma \ref{lem: 3.3}  if and only if any simple  left $R$-module $X$ is flat and $M\otimes_{R}X=0$ by Theorem \ref{thm: 2.10} if and only if  $R$ is a  left $SF$ ring  and $M\otimes_{R}X=0$ for any simple left $R$-module $X$.

(3) $R\ltimes_{n} M$ is a left $V$-ring  if and only if any left $R\ltimes_{n} M$-module $\textbf{Z}(X)$ with $X$ a simple left $R$-module is injective by Lemma \ref{lem: 3.3}  if and only if   any simple  left $R$-module $X$ is injective  and $\Hom_{R}(M,X)=0$ by Theorem \ref{thm: 2.9} if and only if  $R$ is a  left $V$-ring  and $\Hom_{R}(M,X)=0$ for any simple left $R$-module $X$.

(4) $R\ltimes_{n} M$ is a left Kasch ring if and only if $\Hom_{R\ltimes_{n} M}(\textbf{Z}(X), \textbf{T}(R))\neq 0$ for any simple left $R$-module $X$ by \cite[Corollary 8.28]{L} and Lemma \ref{lem: 3.3}. Note that $$\Hom_{R\ltimes_{n}M}(\textbf{Z}(X), \textbf{T}(R))\cong \Hom_{R}(X, \textbf{K}\textbf{T}(R))\cong \Hom_{R}(X, \coprod_{i=0}^{n-1}Ann_{M^{\otimes_{R}i}}(M)\oplus M).$$ So $R$ is a  left Kasch ring if and only if $\Hom_{R}(X, \coprod_{i=0}^{n-1}Ann_{M^{\otimes_{R}i}}(M)\oplus M)\neq 0$  for any simple left $R$-module $X$.
\end{proof}
\bigskip
\section {\bf Some homological formulas of $n$-trivial extensions of rings}
\bigskip
\begin{lem}\label{lem: 4.1} Let $G$ be a right $R$-module and $(X,f)$ a left $R\ltimes_{n}M$-module. Then
\begin{enumerate}\item $\textbf{T}(G)\otimes_{R\ltimes_{n}M}(X,f)\cong G\otimes_{R} X$. \item $\textbf{Z}(G)\otimes_{R\ltimes_{n}M}(X,f)\cong G\otimes_{R}\coker(f)$.
\end{enumerate}
\end{lem}
\begin{proof}(1) $ \textbf{T}(G)\otimes_{R\ltimes_{n}M}(X,f)\cong G\otimes_{R}\textbf{T}(R)\otimes_{R\ltimes_{n}M}(X,f)\cong G\otimes_{R} X$.

(2) $\textbf{Z}(G)\otimes_{R\ltimes_{n}M}(X,f)\cong G\otimes_{R}\textbf{Z}(R)\otimes_{R\ltimes_{n}M}(X,f)\cong G\otimes_{R}\coker(f)$.
\end{proof}
\begin{thm}\label{thm: 4.2}Let $X$ be a left $R$-module, $Y$ a left $R\ltimes_{n}M$-module, $W$ a right $R\ltimes_{n}M$-module and $m\geq 1$.
\begin{enumerate}\item If $\Tor^{R}_{j}(M^{\otimes_{R}i},X)=0$ for any $i,j\geq 1$, then $\Ext^{m}_{R\ltimes_{n}M}(\textbf{T}(X),Y)\cong\Ext^{m}_{R}(X,\textbf{U}(Y))$.\item If $\Ext^{j}_{R}(M^{\otimes_{R}i},X)=0$ for  any $i,j\geq 1$, then $\Ext^{m}_{R\ltimes_{n}M}(Y, \textbf{H}(X))\cong\Ext^{m}_{R}(\textbf{U}(Y),X)$.\item If $\Tor^{R}_{j}(M^{\otimes_{R}i},X)=0$ for any $i,j\geq 1$,  then $\Tor^{R\ltimes_{n}M}_{m}(W, \textbf{T}(X))\cong\Tor^{R}_{m}(\textbf{U}(W),X)$.
\end{enumerate}
\end{thm}
\begin{proof}(1) There is an exact sequence in $R$-Mod $0\rightarrow K_{m}\rightarrow P_{m-1}\rightarrow\cdots\rightarrow P_{1}\rightarrow P_{0}\rightarrow X\rightarrow 0$ with each $P_{k}$ projective. Since $\Tor^{R}_{j}(M^{\otimes_{R}i},X)=0$ for any $i,j\geq 1$,  we get the exact sequence $0\rightarrow\textbf{T}(K_{m})\rightarrow\textbf{T}(P_{m-1})\rightarrow\cdots \rightarrow\textbf{T}(P_{1})\rightarrow\textbf{T}(P_{0})\rightarrow\textbf{T}(X)\rightarrow 0$  with each $ \textbf{T}(P_{k})$ projective. Let $K_{m-1}=\ker(P_{m-2}\rightarrow P_{m-3}), K_{0}=X, K_{1}=\ker(P_{0}\rightarrow X)$. Then the exact sequences $0\rightarrow K_{m}\rightarrow P_{m-1}\rightarrow  K_{m-1}\rightarrow 0$ and $0\rightarrow\textbf{T}(K_{m})\rightarrow\textbf{T}(P_{m-1})\rightarrow\textbf{T}(K_{m-1})\rightarrow 0$ induce the following commutative diagram  $$\xymatrix{\Hom_{R\ltimes_{n}M}(\textbf{T}(P_{m-1}),Y)\ar[d]_{\cong}\ar[r]&\Hom_{R\ltimes_{n}M}(\textbf{T}(K_{m}),Y)\ar[d]_{\cong}
\ar[r]&\Ext_{R\ltimes_{n}M}^{1}(\textbf{T}(K_{m-1}),Y)\ar[d]\ar[r]&0\\
\Hom_{R}(P_{m-1},\textbf{U}(Y))\ar[r]&\Hom_{R}(K_{m},\textbf{U}(Y))\ar[r]&\Ext_{R}^{1}(K_{m-1},\textbf{U}(Y))\ar[r]&0.}$$ Hence $\Ext_{R\ltimes_{n}M}^{1}(\textbf{T}(K_{m-1}),Y)\cong\Ext_{R}^{1}(K_{m-1},\textbf{U}(Y))$. Therefore
$$\Ext_{R\ltimes_{n}M}^{m}(\textbf{T}(X),Y)\cong\Ext_{R\ltimes_{n}M}^{1}(\textbf{T}(K_{m-1}),Y)\cong\Ext_{R}^{1}(K_{m-1},\textbf{U}(Y))\cong\Ext^{m}_{R}(X,\textbf{U}(Y)).$$

The proof of (2) is dual to that of (1).

(3) There is an exact sequence in $R$-Mod $0\rightarrow K_{m}\rightarrow F_{m-1}\rightarrow\cdots\rightarrow F_{1}\rightarrow F_{0}\rightarrow X\rightarrow 0$  with each $F_{k}$ flat. Since $\Tor^{R}_{j}(M^{\otimes_{R}i},X)=0$ for any $i,j\geq 1$,  we get the exact sequence $0\rightarrow\textbf{T}(K_{m})\rightarrow\textbf{T}(F_{m-1})\rightarrow\cdots \rightarrow\textbf{T}(F_{1})\rightarrow\textbf{T}(F_{0})\rightarrow\textbf{T}(X)\rightarrow 0$ with each $ \textbf{T}(F_{k})$ flat. Let $K_{m-1}=\ker(F_{m-2}\rightarrow F_{m-3}), K_{0}=X, K_{1}=\ker(F_{0}\rightarrow X)$. By Lemma \ref{lem: 4.1}(1), the exact sequences $0\rightarrow K_{m}\rightarrow F_{m-1}\rightarrow  K_{m-1}\rightarrow 0$ and $0\rightarrow\textbf{T}(K_{m})\rightarrow\textbf{T}(F_{m-1})\rightarrow\textbf{T}(K_{m-1})\rightarrow 0$ induce the following commutative diagram $$\xymatrix{0\ar[r]&\Tor^{R\ltimes_{n}M}_{1}(W, \textbf{T}(K_{m-1}))\ar[d]\ar[r]&W\otimes_{R\ltimes_{n}M} \textbf{T}(K_{m})\ar[d]_{\cong}\ar[r]&W\otimes_{R\ltimes_{n}M} \textbf{T}(F_{m-1})\ar[d]_{\cong}\\0\ar[r]&\Tor^{R}_{1}(\textbf{U}(W),K_{m-1})\ar[r]&\textbf{U}(W)\otimes_{R}K_{m}\ar[r]&\textbf{U}(W)\otimes_{R}F_{m-1}.}$$ Thus $\Tor^{R\ltimes_{n}M}_{1}(W, \textbf{T}(K_{m-1})) \cong\Tor^{R}_{1}(\textbf{U}(W),K_{m-1}) $. Therefore
$$\Tor^{R\ltimes_{n}M}_{m}(W, \textbf{T}(X))\cong\Tor^{R\ltimes_{n}M}_{1}(W, \textbf{T}(K_{m-1})) \cong\Tor^{R}_{1}(\textbf{U}(W),K_{m-1}) \cong\Tor^{R}_{m}(\textbf{U}(W),X).$$

This completes the proof.
\end{proof}
\begin{cor}\label{cor: 4.3}Let $X$ be a left $R$-module.
\begin{enumerate}\item If $\Tor^{R}_{j}(M^{\otimes_{R}i},X)=0$ for any $i,j\geq 1$,  then $pd(_{R\ltimes_{n}M}\textbf{T}(X))=pd(_{R}X)$ and $fd(_{R\ltimes_{n}M}\textbf{T}(X))=fd(_{R}X)$.\item If $\Ext^{j}_{R}(M^{\otimes_{R}i},X)=0$ for  any $i,j\geq 1$, then $id(_{R\ltimes_{n}M}\textbf{H}(X))=id(_{R}X)$.
\end{enumerate}
\end{cor}
\begin{thm}\label{thm: 4.4}The following conditions are equivalent for a left $R\ltimes_{n}M$-module $(X,f)$:
\begin{enumerate}\item The sequence $M^{\otimes_{R}n+1}\otimes_{R} X\stackrel{(M\otimes f) \cdots(M^{\otimes_{R}n}\otimes f)}\longrightarrow  M\otimes_{R} X\stackrel{f}\longrightarrow X$ is exact. \item  $\Tor^{R\ltimes_{n}M}_{1}(\textbf{Z}(R), (X,f))=0$. \item  $\Tor^{R\ltimes_{n}M}_{1}(\textbf{Z}(G), (X,f))\cong \Tor^{R}_{1}(G, \coker(f))$ for any right $R$-module $G$. \item $\Ext^{1}_{R\ltimes_{n}M}((X,f),\textbf{Z}(Q))\cong \Ext^{1}_{R}(\coker(f),Q)$ for any left $R$-module $Q$.
\end{enumerate}
\end{thm}
\begin{proof}(1) $\Leftrightarrow$ (2)  By Lemma \ref{lem: 2.6}(1), there is an exact sequence in Mod-$R\ltimes_{n}M$ $$0\rightarrow  (\coprod_{i=1}^{n}M^{\otimes_{R}i},h)\rightarrow\textbf{T}(R)\rightarrow\textbf{Z}(R)\rightarrow 0,$$ which induces the exact sequence $$0\rightarrow\Tor^{R\ltimes_{n}M}_{1}(\textbf{Z}(R), (X,f))\rightarrow(\coprod_{i=1}^{n}M^{\otimes_{R}i},h)\otimes_{R\ltimes_{n}M} (X,f)\rightarrow  X\rightarrow\coker(f)\rightarrow 0.$$ Since $(0,0,\cdots,0,M^{\otimes_{R}n})$ is an ideal of $R\ltimes_{n}M$, we get the exact sequence in Mod-$R\ltimes_{n}M$ $$M\otimes_{R}(0,0,\cdots,0,M^{\otimes_{R}n})\rightarrow M\otimes_{R}\textbf{T}(R)\rightarrow M\otimes_{R}(\textbf{T}(R)/(0,0,\cdots,0,M^{\otimes_{R}n}))\rightarrow 0.$$ On the other hand, we have the exact sequence  in Mod-$R\ltimes_{n}M$ $$\textbf{Z}(M^{\otimes_{R}n+1})\rightarrow\textbf{T}(M)\rightarrow(\coprod_{i=1}^{n}M^{\otimes_{R}i},h)\rightarrow 0.$$ So $(\coprod_{i=1}^{n}M^{\otimes_{R}i},h)\cong M\otimes_{R}(\textbf{T}(R)/(0,0,\cdots,M^{\otimes_{R}n}))$. 

Let $f_{n}=f(M\otimes f) \cdots(M^{\otimes_{R}n-1}\otimes f): M^{\otimes_{R}n}\otimes_{R} X\rightarrow X$. Then $$(\coprod_{i=1}^{n}M^{\otimes_{R}i},h)\otimes_{R\ltimes_{n}M} (X,f)\cong M\otimes_{R}(\textbf{T}(R)/(0,0,\cdots,M^{\otimes_{R}n}))\otimes_{R\ltimes_{n}M} (X,f)\cong M\otimes_{R}\coker(f_{n}).$$ Therefore (1) $\Leftrightarrow$ (2) holds by Lemma \ref{lem: 2.7}(1).

(2) $\Rightarrow$ (3) and (2) $\Rightarrow$ (4) There is an exact sequence in $R\ltimes_{n}M$-Mod $$0\rightarrow (K,g)\rightarrow \textbf{T}(P)\rightarrow (X,f)\rightarrow 0$$ with $P$ a projective left $R$-module.
By Lemma \ref{lem: 4.1}(2), we get the exact sequence $$0=\Tor^{R\ltimes_{n}M}_{1}(\textbf{Z}(R), (X,f))\rightarrow\coker(g)\rightarrow P\rightarrow\coker(f)\rightarrow 0.$$
For any right $R$-module $G$, we have the following commutative diagram with exact rows $$\xymatrix{0\ar[r]&\Tor^{R\ltimes_{n}M}_{1}(\textbf{Z}(G), (X,f))\ar[d]\ar[r]&\textbf{Z}(G)\otimes_{R\ltimes_{n}M}(K,g)\ar[d]_{\cong}\ar[r]&\textbf{Z}(G)\otimes_{R\ltimes_{n}M}\textbf{T}(P)\ar[d]_{\cong}\\0\ar[r]&\Tor^{R}_{1}(G, \coker(f))\ar[r]&G\otimes_{R}\coker(g)\ar[r]&G\otimes_{R} P.}$$ Therefore $\Tor^{R\ltimes_{n}M}_{1}(\textbf{Z}(G), (X,f))\cong \Tor^{R}_{1}(G, \coker(f))$.

Let $Q$ be any left $R$-module. Since $(\textbf{C}, \textbf{Z})$ is an adjoint pair, the following commutative diagram with exact rows {\small $$\xymatrix{\Hom_{R\ltimes_{n}M}(\textbf{T}(P),\textbf{Z}(Q))\ar[d]_{\cong}\ar[r]&\Hom_{R\ltimes_{n}M}((K,g),\textbf{Z}(Q))\ar[d]_{\cong}
\ar[r]&\Ext^{1}_{R\ltimes_{n}M}((X,f),\textbf{Z}(Q))\ar[d]\ar[r]&0\\\Hom_{R}( P,Q)\ar[r]&\Hom_{R}(\coker(g),Q)\ar[r]&\Ext^{1}_{R}(\coker(f),Q)\ar[r]&0}$$} implies  that $\Ext^{1}_{R\ltimes_{n}M}((X,f),\textbf{Z}(Q))\cong \Ext^{1}_{R}(\coker(f),Q)$.

(3) $\Rightarrow (2)$  is clear by letting $G=R$.

(4) $\Rightarrow (2)$ By \cite[p.360]{R}, we have $$\Tor_{1}^{R\ltimes_{n}M}(\textbf{Z}(R),(X,f))^{+}\cong\Ext^{1}_{{R\ltimes_{n}M}}((X,f),\textbf{Z}(R)^{+})$$$$\cong\Ext^{1}_{{R\ltimes_{n}M}}((X,f),\textbf{Z}(R^{+}))\cong \Ext^{1}_{R}(\coker(f),R^{+})=0.$$ So $\Tor_{1}^{R\ltimes_{n}M}(\textbf{Z}(R),(X,f))=0.$
\end{proof}
\begin{thm}\label{thm: 4.5}The following conditions are equivalent for a left $R\ltimes_{n}M$-module $(X,f)$:
\begin{enumerate}\item  $\Tor_{i}^{R\ltimes_{n}M}(\textbf{Z}(R),(X,f)) = 0 $ for any $i\geq 1$.\item  $\Tor_{i}^{R\ltimes_{n}M}(\textbf{Z}(G),(X,f))\cong \Tor^{R}_{i}(G,\coker(f))$ for any right $R$-module $G$ and $i\geq 1$.\item  $\Ext^{i}_{R\ltimes_{n}M}((X,f),\textbf{Z}(Q))\cong \Ext^{i}_{R}(\coker(f), Q)$ for any left $R$-module $Q$ and $i\geq 1$.
\end{enumerate}
\end{thm}
\begin{proof}(1) $\Rightarrow$ (2) and (1) $\Rightarrow$ (3) Applying Theorem \ref{thm: 4.4},  it is enough to prove that $\Tor_{i}^{R\ltimes_{n}M}(\textbf{Z}(G),(X,f))\cong \Tor^{R}_{i}(G,\coker(f))$  for any right $R$-module $G$ and $i\geq 2$, and $\Ext^{i}_{R\ltimes_{n}M}((X,f),\textbf{Z}(Q))\cong \Ext^{i}_{R}(\coker(f), Q)$ for any left $R$-module $Q$ and $i\geq 2$.

There is an exact sequence in $R\ltimes_{n}M$-Mod  $$0\rightarrow (K_{i-1},h_{i-1})\rightarrow \textbf{T}(P_{i-2})\rightarrow\cdots\rightarrow\textbf{T}(P_{1})\rightarrow \textbf{T}(P_{0})\rightarrow (X,f)\rightarrow 0$$ with each $P_{k}$ a projective left $R$-module. By (1) and Lemma \ref{lem: 4.1}(2), we obtain the exact sequence $$0\rightarrow\coker(h_{i-1})\rightarrow P_{i-2}\rightarrow\cdots\rightarrow P_{1}\rightarrow P_{0}\rightarrow\coker(f)\rightarrow 0.$$
 Note that $\Tor_{1}^{R\ltimes_{n}M}(\textbf{Z}(R),(K_{i-1},h_{i-1}))\cong\Tor_{i}^{R\ltimes_{n}M}(\textbf{Z}(R),(X,f))=0$.  By Theorem \ref{thm: 4.4}, $$\Tor_{i}^{R\ltimes_{n}M}(\textbf{Z}(G),(X,f))\cong\Tor_{1}^{R\ltimes_{n}M}(\textbf{Z}(G),(K_{i-1},h_{i-1}))$$$$\cong \Tor^{R}_{1}(G,\coker(h_{i-1}))\cong \Tor^{R}_{i}(G,\coker(f)),$$  $$\Ext^{i}_{R\ltimes_{n}M}((X,f),\textbf{Z}(Q))\cong\Ext^{1}_{R\ltimes_{n}M}((K_{i-1},h_{i-1}),\textbf{Z}(Q))$$$$\cong \Ext^{1}_{R}(\coker(h_{i-1}),Q)\cong \Ext^{i}_{R}(\coker(f),Q).$$

(2) $\Rightarrow (1)$  is clear by letting $G=R$.

(3) $\Rightarrow (1)$ Since  $\Ext^{i}_{R\ltimes_{n}M}((X,f),\textbf{Z}(R^{+}))\cong \Ext^{i}_{R}(\coker(f), R^{+})=0$ for any $i\geq 1$, $$\Tor_{i}^{R\ltimes_{n}M}(\textbf{Z}(R),(X,f))^{+}\cong \Ext^{i}_{R\ltimes_{n}M}((X,f),\textbf{Z}(R)^{+})\cong\Ext^{i}_{R\ltimes_{n}M}((X,f),\textbf{Z}(R^{+}))= 0.$$ Thus $\Tor_{i}^{R\ltimes_{n}M}(\textbf{Z}(R),(X,f)) = 0$.
\end{proof}
As a consequence of  Theorem \ref{thm: 4.5} and Lemma \ref{lem: 2.6}(3), we have
\begin{cor}\label{cor: 4.6}If $(X,f)$ is a left $R\ltimes_{n}M$-module such that $\Tor_{i}^{R\ltimes_{n}M}(\textbf{Z}(R),(X,f)) = 0 $ for any $i\geq 1$, then $pd(_{R\ltimes_{n}M}(X,f))=pd(_{R}\coker(f))$ and $fd(_{R\ltimes_{n}M}(X,f))=fd(_{R}\coker(f))$.
\end{cor}
\begin{thm}\label{thm: 4.7}The following conditions are equivalent for a left $R\ltimes_{n}M$-module  $[Y,g]$:
\begin{enumerate}\item The sequence $Y\stackrel{g}\longrightarrow \Hom_{R}(M,Y)\stackrel{{\rm Hom}_{R}^{n}(M,g)\cdots{\rm Hom}_{R}(M,g)}\longrightarrow \Hom^{n+1}_{R}(M,Y)$ is exact.\item $\Ext^{1}_{R\ltimes_{n}M}(\textbf{Z}(R),[Y,g])=0$. \item $\Ext^{1}_{R\ltimes_{n}M}(\textbf{Z}(Q),[Y,g])\cong \Ext^{1}_{R}(Q,\ker(g))$ for any  left $R$-module $Q$.
\end{enumerate}
\end{thm}
\begin{proof}(1) $\Leftrightarrow$ (2) By Lemma \ref{lem: 2.6}(1), there is an exact sequence in $R\ltimes_{n}M$-Mod $$0\rightarrow  (\coprod_{i=1}^{n}M^{\otimes_{R}i},h)\rightarrow\textbf{T}(R)\rightarrow\textbf{Z}(R)\rightarrow 0,$$
 which induces the exact sequence $$0\rightarrow\ker(g)\rightarrow Y\rightarrow  \Hom_{R\ltimes_{n}M}((\coprod_{i=1}^{n}M^{\otimes_{R}i},h),[Y,g])\rightarrow \Ext^{1}_{R\ltimes_{n}M}(\textbf{Z}(R),[Y,g])\rightarrow 0.$$
Note that $(\coprod_{i=1}^{n}M^{\otimes_{R}i},h)\cong(\textbf{T}(R)/(0,0,\cdots,M^{\otimes_{R}n}))\otimes_{R}M$. 

Let $g_{n}=\Hom_{R}^{n-1}(M,g)\cdots\Hom_{R}(M,g)g: Y\rightarrow \Hom_{R}^{n}(M,Y)$. Then $$\Hom_{R\ltimes_{n}M}((\coprod_{i=1}^{n}M^{\otimes_{R}i},h),[Y,g])\cong \Hom_{R\ltimes_{n}M}((\textbf{T}(R)/(0,0,\cdots,M^{\otimes_{R}n}))\otimes_{R}M,[Y,g])$$$$\cong \Hom_{R}(M,\ker(g_{n})).$$ So (1) $\Leftrightarrow$ (2) holds by Lemma \ref{lem: 2.7}(2).

(2) $\Rightarrow$ (3) There is an exact sequence in $R\ltimes_{n}M$-Mod $$0\rightarrow[Y,g]\rightarrow\textbf{H}(E)\rightarrow[L,h]\rightarrow 0$$ with $E$ an injective left $R$-module,
which  induces the exact sequence
$$0\rightarrow\Hom_{R\ltimes_{n}M}(\textbf{Z}(R),[Y,g])\rightarrow\Hom_{R\ltimes_{n}M}(\textbf{Z}(R),\textbf{H}(E))$$$$\rightarrow\Hom_{R\ltimes_{n}M}(\textbf{Z}(R),[L,h])
\rightarrow\Ext^{1}_{R\ltimes_{n}M}(\textbf{Z}(R),[Y,g])= 0.$$ Since $(\textbf{Z}, \textbf{K})$ is an adjoint pair, we get the exact sequence
$0\rightarrow\ker(g)\rightarrow E\rightarrow\ker(h)\rightarrow 0.$

For any  left $R$-module $Q$, we get  the following commutative diagram with exact rows: {\small $$\xymatrix{\Hom_{R\ltimes_{n}M}(\textbf{Z}(Q),\textbf{H}(E))\ar[d]_{\cong}\ar[r]&\Hom_{R\ltimes_{n}M}(\textbf{Z}(Q),[L,h])\ar[d]_{\cong}
\ar[r]&\Ext^{1}_{R\ltimes_{n}M}(\textbf{Z}(Q),[Y,g])\ar[d]\ar[r]&0\\\Hom_{R}(Q, E)\ar[r]&\Hom_{R}(Q,\ker(h))\ar[r]&\Ext^{1}_{R}(Q,\ker(g))\ar[r]&0.}$$} Thus $\Ext^{1}_{R\ltimes_{n}M}(\textbf{Z}(Q),[Y,g])\cong \Ext^{1}_{R}(Q,\ker(g)).$

(3) $\Rightarrow (2)$  is clear by letting $Q=R$.
\end{proof}
\begin{thm}\label{thm: 4.8}The following conditions are equivalent for a left $R\ltimes_{n}M$-module $[Y,g]$:
\begin{enumerate}\item $\Ext_{R\ltimes_{n}M}^{i}(\textbf{Z}(R), [Y,g])=0$ for any $i\geq 1$.\item $\Ext^{i}_{R\ltimes_{n}M}(\textbf{Z}(Q),[Y,g])\cong\Ext^{i}_{R}(Q,\ker(g))$ for any left $R$-module $Q$ and  $i\geq 1$.
\end{enumerate}
\end{thm}
\begin{proof}(1) $\Rightarrow (2)$  By Theorem \ref{thm: 4.7}, it is enough to show that $\Ext^{i}_{R\ltimes_{n}M}(\textbf{Z}(Q),[Y,g])\cong\Ext^{i}_{R}(Q,\ker(g))$ for any left $R$-module $Q$ and $i\geq 2$.

There is an exact sequence in $R\ltimes_{n}M$-Mod $$0\rightarrow [Y,g]\rightarrow \textbf{H}(E^{0})\rightarrow \textbf{H}(E^{1})\rightarrow\cdots\rightarrow \textbf{H}(E^{i-2}) \rightarrow [L^{i-1},h^{i-1}]\rightarrow 0$$ with each $E^{k}$ an injective left $R$-module. By (1), we get the exact sequence {\small $$0\rightarrow\Hom_{R\ltimes_{n}M}(\textbf{Z}(R),[Y,g])\rightarrow\Hom_{R\ltimes_{n}M}(\textbf{Z}(R),\textbf{H}(E^{0}))\rightarrow\Hom_{R\ltimes_{n}M}(\textbf{Z}(R),\textbf{H}(E^{1}))$$$$
\rightarrow\cdots\rightarrow\Hom_{R\ltimes_{n}M}(\textbf{Z}(R),\textbf{H}(E^{i-2}))\rightarrow\Hom_{R\ltimes_{n}M}(\textbf{Z}(R),[L^{i-1},h^{i-1}])\rightarrow 0.$$} Since $(\textbf{Z}, \textbf{K})$ is an adjoint pair, we get the exact sequence
$$0\rightarrow\ker(g)\rightarrow E^{0}\rightarrow E^{1}\rightarrow\cdots\rightarrow E^{i-2}\rightarrow\ker(h^{i-1})\rightarrow 0.$$
Note that $\Ext^{1}_{R\ltimes_{n}M}(\textbf{Z}(R),[L^{i-1},h^{i-1}])\cong\Ext^{i}_{R\ltimes_{n}M}(\textbf{Z}(R),[Y,g])=0$. By Theorem \ref{thm: 4.7},
{\small $$\Ext^{i}_{R\ltimes_{n}M}(\textbf{Z}(Q),[Y,g])\cong\Ext^{1}_{R\ltimes_{n}M}(\textbf{Z}(Q),[L^{i-1},h^{i-1}])\cong\Ext^{1}_{R}(Q,\ker(h^{i-1}))\cong\Ext^{i}_{R}(Q,\ker(g)).$$}
(2) $\Rightarrow (1)$  is clear by letting $Q=R$.
\end{proof}
As a consequence of  Theorem \ref{thm: 4.8} and Lemma \ref{lem: 2.6}(3), we have
\begin{cor}\label{cor: 4.9}If $[Y,g]$ is a left $R\ltimes_{n}M$-module such that $\Ext_{R\ltimes_{n}M}^{i}(\textbf{Z}(R), [Y,g])=0$ for any  $i\geq 1$, then $id(_{R\ltimes_{n}M}[Y,g])=id(_{R}\ker(g))$.
\end{cor}
\bigskip
\bigskip
\centerline {\bf ACKNOWLEDGEMENTS}
\bigskip
This research was supported by NSFC (12271249). The author wants to
express his gratitude to the referee for the very helpful comments
and suggestions.

\end{document}